\documentclass[12pt,reqno]{amsart}

\newcommand\version{November 1, 2016}

\usepackage{amsmath,amsfonts,amsthm,amssymb,amsxtra,enumerate}

\newtheorem{theorem}{Theorem}
\newtheorem{proposition}[theorem]{Proposition}
\newtheorem{lemma}[theorem]{Lemma}
\newtheorem{corollary}[theorem]{Corollary}

\theoremstyle{definition}

\theoremstyle{remark}

\newtheorem{remark}[theorem]{Remark}

\newcommand{\C}{\mathbb{C}}

\newcommand{\dd}{\, \mathrm{d}}

\renewcommand{\epsilon}{\varepsilon}
\newcommand{\e}{\mathrm{e}}
\newcommand{\ii}{\mathrm{i}}
\newcommand{\id}{\mathbb{I}}

\newcommand{\N}{\mathbb{N}}
\newcommand{\norm}[2][]{{\left\|#2\right\|}_{#1}}   

\renewcommand{\phi}{\varphi}
\newcommand{\R}{\mathbb{R}}

\newcommand{\sclp}[2][]{\langle#2\rangle} 
\newcommand{\set}[2]{\left\{#1: \, #2\right\}}
\newcommand{\Sph}{\mathbb{S}}

\DeclareMathOperator{\im}{Im}

\DeclareMathOperator{\re}{Re}

\DeclareMathOperator{\supp}{supp}

\begin{document}

\title[Resolvent estimates --- \version]{Endpoint resolvent estimates\\ for compact Riemannian manifolds}

\author{Rupert L. Frank}
\address{Rupert L. Frank, Mathematisches Institut, Ludwig-Maximilans Universit\"at M\"unchen, Theresienstr. 39, 80333 M\"unchen, Germany, and Mathematics 253-37, Caltech, Pasadena, CA 91125, USA}
\email{rlfrank@caltech.edu}

\author{Lukas Schimmer}
\address{Lukas Schimmer, Mathematics 253-37, Caltech, Pasadena, CA 91125, USA}
\email{schimmer@caltech.edu}

\begin{abstract}
We prove $L^p\to L^{p'}$ bounds for the resolvent of the Laplace--Beltrami operator on a compact Riemannian manifold of dimension $n$ in the endpoint case $p=2(n+1)/(n+3)$. It has the same behavior with respect to the spectral parameter $z$ as its Euclidean analogue, due to Kenig--Ruiz--Sogge, provided a parabolic neighborhood of the positive half-line is removed. This is region is optimal, for instance, in the case of a sphere.
\end{abstract}


\maketitle

\renewcommand{\thefootnote}{${}$} \footnotetext{\copyright\, 2016 by
  the authors. This paper may be reproduced, in its entirety, for
  non-commercial purposes.\\
  The first author is partially supported by U.S. National Science Foundation grant DMS-1363432.
}

\section{Introduction and main result}

In a celebrated work, Kenig, Ruiz and Sogge \cite{Kenig1987} proved $L^p\to L^{p'}$ mapping properties of the resolvent $(-\Delta-z)^{-1}$ in $\R^n$ which are uniform as $z$ approaches $(0,\infty)$. More precisely, they showed that
\begin{equation}
\label{eq:keruso}
\left\|(-\Delta-z)^{-1} f \right\|_{p'} \leq C_{p,n} |z|^{-n/2+n/p-1} \|f\|_p \,,
\qquad 2n/(n+2) \leq p \leq 2(n+1)/(n+3) \,,
\end{equation}
for $n\geq 3$. The same inequality holds for $n=2$ as well, provided the value $2n/(n+2)=1$ for $p$ is excluded. These inequalities and their extensions have found many applications in analysis and PDE, including unique continuation problems and absence of positive eigenvalues \cite{KT1,KT2}, limiting absorption principles \cite{GS,IS}, absolute continuity of the spectrum of periodic Schr\"odinger operators \cite{ZS} and eigenvalue bounds for Schr\"odinger operators with complex potentials \cite{Frank2011}.

Inequalities \eqref{eq:keruso} do \emph{not} hold outside the stated range of $p$. This is easy to see, even for negative values of $z$, at the endpoint $p=2n/(n+2)$. The optimality of $p=2(n+1)/(n+3)$ follows from Knapp's example, once one notices that \eqref{eq:keruso} implies the Stein--Tomas restriction theorem for the Fourier transform (see, e.g., \cite{St}). As a side remark we mention that \eqref{eq:keruso} \emph{can} be extended up to $p<2n/(n+1)$ if the spaces $L^p(\R^n)$ and $L^{p'}(\R^n)$ are replaced by $L^p(\R_+, r^{n-1}dr; L^2(\Sph^{n-1}))$ and $L^{p'}(\R_+, r^{n-1}dr; L^2(\Sph^{n-1}))$, respectively \cite{Frank2015}. A similar extension for the Stein--Tomas theorem is due to \cite{Vega1992}.

Recently, there has been a lot of interest in extending the Kenig--Ruiz--Sogge bounds to compact Riemannian manifolds \cite{Bourgain2015,DosSantos2014,HS,Shao2014}. Dos Santos Ferreira, Kenig and Salo \cite{DosSantos2014} proved that \eqref{eq:keruso} remains valid for $p=2n/(n+2)$, $n\geq 3$ and with $-\Delta$ denoting the Laplace--Beltrami operator, provided one only considers $z\in\C$ with
\begin{equation}
\label{eq:dksregion}
\im\sqrt{z}\geq\delta
\end{equation}
for some arbitrary, but fixed $\delta>0$. Here the branch of the square root on $\C\setminus[0,\infty)$ is chosen such that $\im\sqrt{\cdot}>0$. This restriction on $z$, which can be written as $(\im z)^2 \geq 4\delta^2(\re z+\delta^2)$, excludes a neighborhood of the origin and a parabolic region around $(0,\infty)$. Shortly afterwards, Bourgain, Shao, Sogge and Yao \cite{Bourgain2015} showed that for some manifolds (including spheres) the restriction \eqref{eq:dksregion} is  necessary for the inequality to hold, whereas for other manifolds (for instance, tori or manifolds with non-positive sectional curvatures) the inequality holds in a larger region.

It is not difficult to see that the proof in \cite{DosSantos2014} actually gives the analogue of \eqref{eq:keruso} for any $2n/(n+2) \leq p < 2(n+1)/(n+3)$ if $n\geq 3$ and any $z$ satisfying \eqref{eq:dksregion}. Moreover, their proof extends to $n=2$ and $1=2n/(n+2) < p < 2(n+1)/(n+3)=6/5$. (We discuss this in Remark \ref{nonendpoint} and the appendix in more detail.) The argument of \cite{DosSantos2014}, however, does \emph{not} cover the endpoint $p = 2(n+1)/(n+3)$, and our contribution in this short note is to provide a proof of the following theorem.

\begin{theorem}\label{th:Res}
Let $M$ be a compact Riemannian manifold without boundary of dimension $n\geq 2$ and let $\delta>0$. Then there is a constant $C$ such that for all $f\in L^{2(n+1)/(n+3)}(M)$ and all $z\in\C$ with $\im\sqrt{z}\geq\delta$,
\begin{equation}
\label{eq:main}
\left\|(-\Delta-z)^{-1} f \right\|_{2(n+1)/(n-1)} \leq C |z|^{-\frac1{n+1}} \|f\|_{2(n+1)/(n+3)} \,.
\end{equation}
\end{theorem}

We can also show that for a certain class of manifolds (including spheres) the region \eqref{eq:dksregion} cannot be significantly extended. Recall that a Zoll manifold is a Riemannian manifold for which the geodesic flow is periodic with a common minimal period.

\begin{proposition}\label{th:opt}
Let $M$ be a compact Riemannian manifold without boundary of dimension $n\geq 2$ which is Zoll. Then there is a constant $C> 0$ such that for any function $\delta:\R\to(0,\infty)$ with $\lim_{|\kappa|\to\infty} \delta(\kappa)=0$ and $\liminf_{|\kappa|\to\infty} |\kappa| \delta(\kappa) \geq C$,
\begin{equation*}
\limsup_{|\kappa|\to\infty}  \left|\left(\kappa+\ii\delta(\kappa)\right)^2\right|^{\frac1{n+1}} \left\| \left(-\Delta- (\kappa+\ii\delta(\kappa))^2 \right)^{-1} \right\|_{L^{2(n+1)/(n+3)}\to L^{2(n+1)/(n-1)}} = \infty \,.
\end{equation*}
If $M=\Sph^n$ (with the standard metric), this holds also with $C=0$.
\end{proposition}

The proof of this proposition follows rather closely the arguments in \cite{Bourgain2015}. It is based on the optimality of Sogge's spectral cluster estimates \cite{Sogge1988} and Weinstein's theorem about the clustering of eigenvalues on Zoll manifolds \cite{Weinstein1977}. The basic idea is that the size of $(-\Delta-z)^{-1}$ is dominated by the spectrum of $-\Delta$ in a neighborhood of $\re z$ of size $\im z$. This can be made rigorous by showing that the resolvent inequality \eqref{eq:main} in the region \eqref{eq:dksregion} implies mapping properties of the projection operator onto the eigenspaces of the Laplacian corresponding to eigenvalues in $[\re z-\delta\sqrt{\re z},\re z+\delta\sqrt{\re z}]$ (`spectral cluster'). When $\delta$ is a positive constant, this is precisely the window in which Sogge's spectral cluster estimates are valid. If the resolvent bounds hold with $\delta$ depending on $\re z$ and tending to zero as $\re z\to\infty$, one obtains mapping properties of the projection onto significantly smaller spectral clusters. Due to the strong concentration of eigenvalues on Zoll manifolds, however, for appropriately chosen $\re z$, the projection onto these smaller spectral clusters coincides with that on the regular spectral clusters. Therefore the optimality of Sogge's estimates shows that there cannot be an improvement to smaller spectral clusters on Zoll manifolds. The details of this argument can be found in Section \ref{sec:opt}.

Let us discuss some ingredients in our proof of Theorem \ref{th:Res}. We will follow the general strategy in \cite{DosSantos2014} which is based on a Hadamard parametrix for $(-\Delta-z)^{-1}$. The arguments of \cite{DosSantos2014} show that the $L^p\to L^{p'}$ mapping property for the parametrix imply the same for $(-\Delta-z)^{-1}$. The proof of the mapping properties for the parametrix in \cite{DosSantos2014}, however, does not extend to the case $p=2(n+1)/(n+3)$, since the corresponding geometric series in \cite[Subsec. 4.1]{DosSantos2014} diverge. (This is somewhat reminiscent of Tomas' original proof of the Stein--Tomas theorem \cite{Tomas1975}.) We overcome this impasse by proving an off-diagonal $L^{p,1}\to L^{q,\infty}$ bound for the parametrix, where $1/p-1/q=2/(n+1)$ and where $L^{r,s}$ denotes the scale of Lorentz spaces. By real interpolation between this and its dual we obtain the desired $L^{2(n+1)/(n+3)}\to L^{2(n+1)/(n-1)}$ bound. (Note that $(n+3)/(2(n+1)) - (n-1)/(2(n+1)) = 2/(n+1)$.) As a small bonus of this method, we even obtain the slightly stronger $L^{2(n+1)/(n+3),2}\to L^{2(n+1)/(n-1),2}$ bound for the parametrix.

The idea of proving endpoint inequalities in Lorentz spaces is an old one in harmonic analysis, as is the observation that for bounds from $L^{p,1}$ to some Banach space one can restrict oneself to characteristic functions. In the context of uniform Sobolev inequalities this technique seems to have been used first by Guti{\'e}rrez  \cite{Gutierrez2004}. In the context of Fourier restriction theorems it appeared recently in \cite{Bak2011} and in our context of the Laplace--Beltrami operator on compact manifolds in \cite{Sogge2016}.


\section{Hadamard parametrix}\label{sec:Par}

We will use the Hadamard parametrix $T(z)$ as in \cite{DosSantos2014}. The explicit construction will not be relevant for us (we recall some details about it in our appendix) and we will summarize here all properties that we make use of. Throughout this section we assume that $z\in\C\setminus[0,\infty)$ and our bounds will be uniform as $z$ approaches the positive half-line, as long as a neighborhood of the origin is avoided. In fact, our construction also works for $z\in (0,\infty)$ with the understanding that the two values $z\pm\ii0$ are different.

The construction is local and we fix a point $x_0\in M$ and denote by $U$ the geodesic ball around $x_0$ with radius equal to the injectivity radius of $M$. We consider two functions $\chi,\tilde\chi\in C^\infty_0(U)$ such that $\tilde\chi\equiv 1$ in a neighborhood of $\supp\chi$. The Hadamard parametrix is the operator $T(z)$ defined by
\begin{align}
(T(z)u)(x)=\int_M\tilde \chi(x)F(x,y,z) \chi(y) u(y)\dd\mu_g(y)
\label{eq:defT}
\end{align}
with $F$ from \cite{DosSantos2014} and with $\dd\mu_g$ denoting the volume form on $M$.

As explained in the introduction, to obtain bounds on the parametrix on the line $\frac1p-\frac1q=\frac{2}{n+1}$ we will prove weak estimates on the endpoints of this line and then use real interpolation, similarly as \cite{Gutierrez2004} in Euclidean space.
 
\begin{proposition}\label{prop:weak}
Let $\delta>0$ and either $(p,q)=(\frac{2n(n+1)}{n^2+4n-1},\frac{2n}{n-1})$ or $(p,q)=(\frac{2n}{n+1},\frac{2n(n+1)}{n^2-2n+1})$. Then, if $|z|\geq \delta$,
\begin{align}
\norm[q,\infty]{T(z)u}\lesssim |z|^{-\frac{1}{n+1}}\norm[p,1]{u} \,.
\label{eq:Parweak1} 
\end{align}
\end{proposition}

Here $\norm[r,s]{\cdot}$ denotes the norm on the  Lorentz space $L^{r,s}(M)$ defined as
\begin{align*}
\norm[r,s]{u}=\begin{cases}
r^\frac1s\left(\int_0^\infty \lambda^{s-1}\mu_g(\set{x\in M}{|u(x)|\ge\lambda})^\frac{s}{r}\dd\lambda\right)^{1/s}
&\text{if } 0<r<\infty, 0<s<\infty \,,\\
\sup_{\lambda>0}\left(\lambda\mu_g(\set{x\in M}{|u(x)|>\lambda})^{\frac1r}\right)
&\text{if } 0<r<\infty, s=\infty\,.
\end{cases}
\end{align*}

Real interpolation (see, e.g., \cite[Chapter V, Theorem 3.15]{Stein1971}) between the weak estimates of Proposition \ref{prop:weak}  yields $L^p\to L^q$ bounds for values of $p,q$ that lie on the line $\frac1p-\frac1q=\frac{2}{n+1}$ between the two extremal points. The results are summarized as follows.

\begin{corollary}\label{th:Par}
Let $\delta>0$ and let $1\le p\le2\le q$ with
\begin{align*}
\frac1p-\frac1q=\frac{2}{n+1}\,,\quad \frac{n^2-2n+1}{2n(n+1)}<\frac1q< \frac{n-1}{2n} \,.
\end{align*}
Then, if $|z|\geq\delta$,
\begin{align*}
\norm[q]{T(z)u}\lesssim |z|^{-\frac{1}{n+1}}\norm[p]{u} \,.
\end{align*}
\end{corollary}

\begin{remark}
In fact, the interpolation result of \cite{Stein1971} yields the inequality
\begin{align*}
\norm[q,s]{T(z)u}\lesssim |z|^{-\frac{1}{n+1}}\norm[p,s]{u}\,.
\end{align*}
for $1\le s\le \infty$, which for $p<s<q$ is stronger than Corollary~\ref{th:Par}. 
\end{remark}

Our proof of Proposition \ref{prop:weak} relies on bounds on the Hadamard parametrix from \cite{DosSantos2014} or, more precisely, on bounds on pieces of it in a dyadic decomposition. To state these, let $\psi_0\in\mathcal{C}^\infty_0(\R)$ be supported in $[-1,1]$ and equal to $1$ on $[-1/2,1/2]$. For $\nu\ge 1$ we define $\psi_\nu(r)=\psi_0(2^{-\nu-1}r)-\psi_0(2^{-\nu}r)\in\mathcal{C}_0^\infty(\R)$, which is supported in $[-2^{\nu},-2^{\nu-1}]\cup[2^{\nu-1},2^{\nu}]$. This construction yields a dyadic partition of unity
\begin{align}
1=\sum_{\nu\ge 0}\psi_{\nu}(r)\,.
\label{eq:dyadicsum} 
\end{align} 
With the distance function $d_g(x,y)$ on $U\times U$ we define for $\nu\ge0$ the integral operators
\begin{align}
(T_\nu(z)u)(x)=\int_M\tilde\chi(x) \psi_\nu\big(|z|^\frac12 d_g(x,y)\big)F(x,y,z) \chi(y) u(y)\dd \mu_g(y) \,.
\label{eq:defTnu} 
\end{align}
Because of \eqref{eq:dyadicsum} we have
\begin{align*}
T(z)=\sum_{\nu\ge 0} T_\nu(z)\,.
\end{align*}
With these definitions, we can state the following bounds.

\begin{lemma}\label{lem:dyadic}
Let $\delta>0$. Then the following holds for $|z|\geq\delta$. 
\begin{enumerate}[a)]
\item For $\nu=0$ the kernel of the integral operator $T_0(z)$ can be bounded by
\begin{align}
C\left( 1+ \log(|z|^{-1/2}d_g(x,y)^{-1}) \right) \id_{|z|^{1/2}d_g(x,y)\le 1}
\qquad\text{if}\ n=2
\label{eq:T0log} 
\end{align}
and by
\begin{align}
Cd_g(x,y)^{2-n}\id_{|z|^{1/2}d_g(x,y)\le 1}
\qquad\text{if}\ n\ge 3 \,.
\label{eq:T0} 
\end{align}

\item For $\nu\ge1$ and $n\ge2$ the norm of the integral operator $T_\nu(z)$ can be bounded by
\begin{align}
\norm[q]{T_\nu(z)u}\le C |z|^{\frac{n}{2p}-\frac{n}{2q}-1}2^{-\nu(\frac{n}{p}-\frac{n+1}{2})}\norm[p]{u}
\quad \text{if}\quad 1\le p \le 2,\ q=\frac{n+1}{n-1}p'\,,
\label{eq:Tnud} 
\end{align}
as well as the dual
\begin{align}
\norm[q]{T_\nu(z)u}\le C |z|^{\frac{n}{2p}-\frac{n}{2q}-1}2^{-\nu(\frac{n}{q'}-\frac{n+1}{2})}\norm[p]{u}
\quad \text{if}\quad  2\le q\le\infty,\ p'=\frac{n+1}{n-1}q\,.
\label{eq:Tnu}
\end{align}

\end{enumerate}

\end{lemma}

In case $n\geq 3$ all three bounds can be found in \cite{DosSantos2014}. The first statement appears before \cite[(4.5), see also (3.10)]{DosSantos2014}. For $\arg z\in[-\pi/2,\pi/2]$ the second bound is equation \cite[(4.11)]{DosSantos2014} and results from the  Carleson--Sj{\"o}lin theorem (see e.g. \cite[Corollary 2.2.3]{Sogge1993}). For $\arg z\notin[-\pi/2,\pi/2]$ the bounds can be improved with an exponential decay, as shown in the paragraph after equation \cite[(4.6)]{DosSantos2014}. As we shall show in the appendix, the proofs of these results in \cite{DosSantos2014} extend to the case $n=2$.

\begin{proof}[Proof of Proposition \ref{prop:weak}]
By duality we need only consider $(p,q)=(\frac{2n(n+1)}{n^2+4n-1},\frac{2n}{n-1})$. 

\emph{Step 1.} We first prove the estimate for $T_0(z)$. We use the following simple generalization of Young's inequality, which says that when $1+1/q=1/p+1/r$, then
$$
\left|\iint v(x)k(x,y)u(y)\dd x\dd y \right| \leq \text{ess-}\sup_x \|k(x,\cdot)\|_r^{\frac{r}{q'}-\frac{r}{r'}} \text{ess-}\sup_y \|k(\cdot,y)\|_r^{\frac{r}{p}-\frac{r}{r'}} \|v\|_{q'} \|u\|_p \,.
$$
Let $r:= (n+1)/(n-1)$ and note that for the given $p$ and $q$ we have, indeed, $1+1/q=1/p+1/r$. Thus, by \eqref{eq:T0} and \eqref{eq:T0log},
$$
\norm[q]{T_0(z)u} \lesssim \left( \sup_{x\in U} \int_{|z|^{1/2}d_g(x,y)\leq 1} |z|^{r(n-2)/2}\ k(|z|^{1/2} d_g(x,y))^r \dd\mu_g(y) \right)^{1/r} \|u\|_p \,,
$$
where $k(t)=t^{2-n}$ if $n\geq 3$ and $k(t) = 1+\log(1/t)$ if $n=2$. Since $r<2/(n-2)$, the integral is finite and, since in local coordinates the distance is comparable to the Euclidean distance, we easily compute that for $|z|\geq\delta$
$$
\left( \sup_{x\in U} \int_{|z|^{1/2}d_g(x,y)\leq 1} |z|^{r(n-2)/2} k(|z|^{1/2} d_g(x,y))^r \dd\mu_g(y) \right)^{1/r} \lesssim |z|^{-1/(n+1)}.
$$
Thus, we obtain the claimed bound even with the usual (instead of the Lorentz) norms.

\emph{Step 2.} We now bound $\widetilde{T}(z)=T(z)-T_0(z)$. As proved in \cite[Chapter V, Theorem 3.13]{Stein1971} it is sufficient to show the analogue of \eqref{eq:Parweak1} for $u$ equal to a characteristic function $\id_E$ of a measurable set $E\subset M$. The statement is then equivalent to
\begin{align*}
\left(\sup_{\lambda>0}\lambda^q\mu_g\big(\big\{x\in M:\,\big|\big(\widetilde{T}(z)\id_E\big)(x)\big|>\lambda\big\}\big)\right)^{1/q}\le C |z|^{-\frac{1}{n+1}}\mu_g(E)^{1/p}\,.
\end{align*}
Let us fix $E\subset M$. For an integer $\rho\in\N_0$ to be specified later (depending on $E$ and $|z|$), we write $\tilde T(z) = T^{(1)}(z) + T^{(2)}(z)$ with
$$
T^{(1)}(z) = \sum_{1\leq\nu\leq\rho} T_\nu(z) \,,
\qquad
T^{(2)}(z) = \sum_{\nu>\rho} T_\nu(z) \,.
$$
We abbreviate $A=\big\{x\in M:\,\big|\big(\widetilde{T}(z)\id_E\big)(x)\big|>\lambda\big\}$ and bound
\begin{align}
\begin{split}
\mu_g(A)&=\int_A\dd\mu_g(x)\le\frac1\lambda\int_A\big|\big(\widetilde{T}(z)\id_E\big)(x)\big|\dd \mu_g(x)\\
&\le \frac1\lambda\int_A\left|\left(T^{(1)}(z)\id_E\right)(x)\right|\dd \mu_g(x)+ \frac1\lambda\int_A \left|\left(T^{(2)}(z)\id_E\right)(x)\right|\dd \mu_g(x) \\
&\le \frac1\lambda \left( \|T^{(1)}(z)\id_E\|_{q_1} \mu_g(A)^{\frac{1}{q_1'}} + \|T^{(2)}(z)\id_E\|_{q_2} \mu_g(A)^{\frac{1}{q_2'}} \right)
\end{split}
\label{eq:split} 
\end{align}
with $1\le q_1,q_2\leq\infty$ to be determined. 

The key observation now is that for the pieces $T^{(1)}(z)$ and $T^{(2)}(z)$ the $L^p\to L^q$ estimates hold in a larger range of parameters $p,q$. The operator $T^{(1)}(z)$ consists of finitely many $T_\nu(z)$ and the bounds of Lemma \ref{lem:dyadic} for $q=2$ can then be summed over.  For $T^{(2)}(z)$ only bounds on $T_\nu(z)$ for large $\nu$ need to be considered, and in the case  $q=\infty$ these form a geometric series. 

To be more precise, let $p_1=2(n+1)/(n+3), q_1=2$. Using \eqref{eq:T0} and \eqref{eq:Tnu} we obtain
\begin{align*}
\|T^{(1)}(z)u\|_{q_1}&\le\sum_{1\leq\nu\leq\rho}\norm[q_1]{T_\nu(z)u}
\lesssim\norm[p_1]{u}\sum_{1\leq\nu\leq\rho}|z|^{-\frac{n+2}{2(n+1)}}2^{\frac{\nu}{2}}\\
&\lesssim |z|^{-\frac{n+2}{2(n+1)}} 2^\frac\rho2\id_{\rho\geq 1} \norm[p_1]{u}\,.
\end{align*}
For $p_2=1,q_2=\infty$ we use 
\eqref{eq:Tnud} as well as the convergence of the geometric series to bound
\begin{align*}
\|T^{(2)}(z)u\|_{q_2}
&\le 
\sum_{\nu\geq \rho+1} \norm[q_2]{T_\nu(z)u}
\lesssim 
\norm[p_2]{u}\sum_{\nu\geq \rho+1} |z|^{\frac{n-2}{2}}2^{-\nu\frac{n-1}{2}}\\
&\lesssim  |z|^{\frac{n-2}{2}} 2^{-\frac{(n-1)\rho}{2}}\norm[p_2]{u}\,.
\end{align*}

Inserting these bounds into \eqref{eq:split} we are arrive at
\begin{align}
\mu_g(A)
\le \frac{C}{\lambda}\left(
|z|^{-\frac{n+2}{2(n+1)}} \mu_g(E)^{\frac{1}{p_1}}\mu_g(A)^{\frac{1}{q_1'}} 2^{\frac{\rho}{2}} \id_{\rho\geq1}
+|z|^{\frac{n-2}{2}}\mu_g(E)^{\frac{1}{p_2}}\mu_g(A)^{\frac{1}{q_2'}} 2^{-\frac{(n-1)\rho}{2}}
\right).
\label{eq:gutierrez} 
\end{align}
It remains to optimize over the choice of $\rho$. If we could vary continuously over $R=2^\rho$ the minimum would be achieved at (a universal constant times)
$$
R_0 := |z|^{\frac{n}{n+1}}\mu_g(E)^{\frac{2}{np_1'}}\mu_g(A)^{\frac{2}{nq_1}},
$$
and we would get the desired bound
\begin{align*}
\lambda \mu_g(A)^{\frac{n-1}{2n}}\le C |z|^{-\frac{1}{n+1}}\mu_g(E)^{\frac{n^2+4n-1}{2n(n+1)}} \,.
\end{align*}
Since our $R$ has to satisfy certain restrictions, we have to argue slightly more carefully. If $R_0> 1$, then we choose $\rho\in\N_0$ such that $2^\rho < R_0 \le 2^{\rho+1}$. Inserting this into \eqref{eq:gutierrez} we obtain again the desired bound. On the other hand, if $R_0\leq 1$, then we choose $\rho=0$. Consequently, the first term on the right side of \eqref{eq:gutierrez} disappears and we obtain
\begin{align*}
\mu_g(A)
\le \frac{C}{\lambda} |z|^{\frac{n-2}{2}} \mu_g(E)^{\frac{1}{p_2}}\mu_g(A)^{\frac{1}{q_2'}}
= \frac{C}{\lambda} |z|^{-\frac{1}{n+1}}\mu_g(E)^{\frac{n^2+4n-1}{2n(n+1)}} \mu_g(A)^{\frac{n+1}{2n}} R_0^\frac{n-1}{2} \,.
\end{align*}
Since $R_0\leq 1$, we obtain again the desired bound.  This proves the proposition.
\end{proof}

We conclude this section by citing two more results from \cite{DosSantos2014} about the Hadamard parametrix and its remainder $S(z)$, defined by
\begin{align}
(-\Delta_g-z)T(z)u=\chi(x)u+S(z)u\,,
\label{eq:T} 
\end{align}
which will be important in our proof of Theorem \ref{th:Res} in the next section.

\begin{lemma}\label{lem:Tadd}
Let $\delta>0$. Then, for all $|z|\geq \delta$,
\begin{align*}
\norm[2]{T(z)u}\lesssim  |z|^{-\frac{n+3}{4(n+1)}} \norm[\frac{2(n+1)}{n+3}]{u} \,.
\end{align*}
\end{lemma}

\begin{lemma}\label{lem:S}
Let $\delta>0$. Then, for all $|z|\geq \delta$,
\begin{align*}
\norm[2]{S(z)u}\lesssim |z|^{\frac{n-1}{4(n+1)}} \norm[\frac{2(n+1)}{n+3}]{u} \,.
\end{align*}
\end{lemma}

These bounds appear as \cite[Lemma 4.1]{DosSantos2014} and \cite[Lemma 4.2]{DosSantos2014}, respectively, for $n\ge 3$. We elaborate upon the case $n=2$ in the appendix.


\section{Resolvent estimates}

To construct a global parametrix, we cover the compact manifold $M$ by finitely many points such that the geodesic balls $U_1,\ldots,U_J$ around these points with radius equal to the injectivity radius of $M$ cover the manifold. We apply the parametrix construction from the previous section around each of these points. The functions $\chi_j\in C^\infty_0(U_j)$ are chosen such that $\sum_{j=1}^J\chi_j=1$. As in the previous subsection we also need $\tilde\chi_j\in C^\infty_0(U_j)$ such that $\tilde\chi_j=1$ in a neighborhood of $\supp\chi_j$. We denote by $T_j(z)$ the Hadamard parametrix of Section \ref{sec:Par} on $U_{j}$ with cut-off function $\tilde{\chi}_j(x)\chi_j(y)$ and by $S_j(z)$ the corresponding remainder term in \eqref{eq:T}. 

We note that our choice of cut-off functions differs from that in \cite{DosSantos2014} and is dictated by the proof of Lemma \ref{lem:S} where $[\Delta_x,\tilde\chi(x)\chi(y)]$ needs to be supported away from the diagonal.

The global parametrix is defined as 
\begin{align*}
\mathcal{T}(z)=\sum_{j=1}^JT_j(z)\,. 
\end{align*}
Setting 
\begin{align*}
\mathcal{S}(z)=\sum_{j=1}^J S_j(z)\,,
\end{align*}
we see from \eqref{eq:T} and $\sum_j \chi_j =1$ that
\begin{align}
(-\Delta_g-z) \mathcal{T}(z)=\id +\mathcal{S}(z) \,.
\label{eq:IdS} 
\end{align}

We can now finish the proof of Theorem \ref{th:Res}. 

\begin{proof}[Proof of Theorem \ref{th:Res}]
We use \eqref{eq:IdS} to write
$$
(-\Delta_g-z)^{-1} = \mathcal T(z) - (-\Delta_g-z)^{-1} \mathcal S(z) \,.
$$
Moreover, taking the adjoint of \eqref{eq:IdS} with $z$ replaced by $\overline z$ we obtain
$$
(-\Delta_g-z)^{-1} = \mathcal T(\overline z)^* - \mathcal S(\overline z)^* (-\Delta_g-z)^{-1} \,,
$$
and inserting this into the first equation we obtain
$$
(-\Delta_g-z)^{-1} = \mathcal T(z) - \mathcal T(\overline z)^* \mathcal S(z) + \mathcal S(\overline z)^* (-\Delta_g-z)^{-1} \mathcal S(z) \,.
$$
Therefore, we can bound
\begin{align}
\label{eq:resbound}
\left\| (-\Delta_g-z)^{-1} \right\|_{L^p\to L^{p'}} & \le \left\| \mathcal T(z) \right\|_{L^p\to L^{p'}} + \left\| \mathcal T(\overline z)^* \right\|_{L^2\to L^{p'}} \left\| \mathcal S(z) \right\|_{L^p\to L^2} \notag \\
& \qquad + \left\| \mathcal S(\overline z)^* \right\|_{L^2\to L^{p'}} \left\| (-\Delta_g-z)^{-1} \right\|_{L^2\to L^{2}} \left\| \mathcal S(z) \right\|_{L^p\to L^2} \notag \\
& \le \left\| \mathcal T(z) \right\|_{L^p\to L^{p'}} + \left\| \mathcal T(\overline z) \right\|_{L^p\to L^2} \left\| \mathcal S(z) \right\|_{L^p\to L^2} \notag \\
& \qquad + \left\| \mathcal S(\overline z) \right\|_{L^p\to L^2} \left\| (-\Delta_g-z)^{-1} \right\|_{L^2\to L^{2}} \left\| \mathcal S(z) \right\|_{L^p\to L^2} \,.
\end{align}
We apply this with $p=2(n+1)/(n+3)$. According to the parametrix estimates of Theorem \ref{th:Par} and Lemma \ref{lem:Tadd} for each $T_j(z)$, we have
$$
\left\| \mathcal T(z) \right\|_{L^p\to L^{p'}} \lesssim |z|^{-1/(n+1)}
\qquad\text{and}\qquad
\left\| \mathcal T(z) \right\|_{L^p\to L^2} \lesssim |z|^{-(n+3)/(4(n+1))} \,.
$$
On the other hand, by Lemma \ref{lem:S} for each $S_j(z)$,
$$
\left\| \mathcal S(z) \right\|_{L^p\to L^2} \lesssim |z|^{(n-1)/(4(n+1))} \,.
$$
Note that these bounds are all valid for $|z|\geq\delta$. The assumption $\im\sqrt z\geq\delta$ only comes in when bounding $(-\Delta_g-z)^{-1}$ on $L^2$, which we do as follows. Since the eigenvalues of $-\Delta_g$ are all non-negative and since $|\lambda-z|\ge\sqrt{|z|}\im\sqrt{z}$ for all $\lambda\ge0$ and $z\in\C\setminus(0,\infty)$, an application of the functional calculus yields
\begin{align*}
\left\| (-\Delta_g-z)^{-1} \right\|_{L^2\to L^{2}} \le|z|^{-\frac12}(\im\sqrt{z})^{-1} \,.
\end{align*}
Thus, $\left\| (-\Delta_g-z)^{-1} \right\|_{L^2\to L^{2}}\le |z|^{-1/2}\delta^{-1}$ for $\im\sqrt z\geq\delta$. Inserting these operator norm bounds into \eqref{eq:resbound} we finally obtain the theorem. 
\end{proof}

\begin{remark}\label{nonendpoint}
By the same argument one can show that, if $M$ is a compact Riemannian manifold without boundary of dimension $n\geq 2$, if $2n/(n+2)<p<2(n+1)/(n+3)$ and if $\delta>0$, there is a constant $C$ such that for all $f\in L^p(M)$ and all $z\in\C$ with $\im\sqrt{z}\geq\delta$,
\begin{equation*}
\left\|(-\Delta_g-z)^{-1} f \right\|_{p'} \leq C |z|^{-n/2+n/p-1} \|f\|_{p} \,.
\end{equation*}
In fact, if $n\geq 3$, then
$$
\left\| \mathcal T(z) \right\|_{L^p\to L^{p'}} \lesssim |z|^{-n/2+n/p-1}
\qquad\text{and}\qquad
\left\| \mathcal T(z) \right\|_{L^p\to L^2} \lesssim |z|^{(n-3)/4 - n/(2p')}
$$
by \cite[Theorem 4.1 and Lemma 4.1]{DosSantos2014} and
$$
\left\| \mathcal S(z) \right\|_{L^p\to L^2} \lesssim |z|^{(n-1)/4 - n/(2p')}
$$
by \cite[Lemma 4.2]{DosSantos2014}. These bounds remain valid for $n=2$, as shown in the appendix.
\end{remark}


\section{Optimality for Zoll manifolds}\label{sec:opt}

Our proof of Proposition \ref{th:opt} is based on the following comparison between resolvent norm bounds and spectral cluster norm bounds.

\begin{lemma}\label{lem:Pkd}
Let $A$ be a self-adjoint real operator in $L^2$ and $1\leq p \leq\infty$. Let $\kappa,\delta>0$ and set $P_{\kappa,\delta}=\id_{[(\kappa-\delta)^2,(\kappa+\delta)^2]}(A)$. Then
\begin{align*}
\norm[L^p\to L^2]{P_{\kappa,\delta}}^2
\le 4\delta\kappa \left(1+\delta\kappa^{-1}+(1/2)\delta^2\kappa^{-2}\right) \norm[L^p\to L^{p'}]{\left(A-(\kappa+\ii\delta)^2\right)^{-1}}\,.
\end{align*}
\end{lemma}

\begin{proof}
Of course we may assume that the resolvent on the right side is bounded, for otherwise there is nothing to prove. Since $A$ is real, we have
\begin{align*}
\norm[L^p\to L^{p'}]{\left(A-(\kappa+\ii\delta)^2\right)^{-1}-\left(A-(\kappa-\ii\delta)^2\right)^{-1}}
\le 2\norm[L^p\to L^{p'}]{\left(A-(\kappa+\ii\delta)^2\right)^{-1}} \,.
\end{align*}
Since for any $\lambda\in\R$
\begin{align*}
\frac{1}{\lambda-(\kappa+\ii\delta)^2}-\frac{1}{\lambda-(\kappa-\ii\delta)^2}
=\frac{4\ii\delta\kappa}{(\lambda-\kappa^2+\delta^2)^2+4\delta^2\kappa^2},
\end{align*}
we can rewrite this, using the functional calculus, as
\begin{align}
\norm[L^p\to L^{p'}]{4\delta\kappa\left((A-\kappa^2+\delta^2)^2+4\delta^2\kappa^2\right)^{-1}}
\le 2\norm[L^p\to L^{p'}]{\left(A-(\kappa+\ii\delta)^2\right)^{-1}}\,.
\label{eq:TTeps} 
\end{align}
The operator on the left-hand side is non-negative and can be written as $T^*T$ with
\begin{align*}
T=\sqrt{4\delta\kappa}\left((A-\kappa^2+\delta^2)^2+4\delta^2\kappa^2\right)^{-\frac12}\,,
\end{align*}
It follows from \eqref{eq:TTeps} that $T$ is bounded from $L^p\to L^2$ with
\begin{align}
\norm[L^p\to L^2]{T}^2\le 2\norm[L^p\to L^{p'}]{\left(A-(\kappa+\ii\delta)^2\right)^{-1}}
\,.
\label{eq:Teps} 
\end{align}
Note that for $\lambda\in[(\kappa-\delta)^2,(\kappa+\delta)^2]$ we can bound $|\lambda-\kappa^2+\delta^2|\le 2\delta\kappa+2\delta^2$ and consequently, by the functional calculus, on $L^2$
\begin{align*}
P_{\kappa,\delta}\le 2\delta\kappa \left(1+\delta\kappa^{-1}+(1/2)\delta^2\kappa^{-2}\right) T^2\,
\end{align*}
Since $P_{\kappa,\delta}$ is a projection we can conclude that for $u\in L^2\cap L^p$
\begin{align*}
\norm[2]{P_{\kappa,\delta}u}^2
=\sclp{P_{\kappa,\delta}u,u}
&\le 2\delta\kappa \left(1+\delta\kappa^{-1}+(1/2)\delta^2\kappa^{-2}\right) \sclp{Tu,Tu}\\
&\le  2\delta\kappa \left(1+\delta\kappa^{-1}+(1/2)\delta^2\kappa^{-2}\right) \norm[p]{u}^2\norm[L^p\to L^2]{T}^2\,.
\end{align*}
This, together with \eqref{eq:Teps} and the density of $L^2\cap L^p$ in $L^p$ imply the lemma.
\end{proof}

\begin{proof}[Proof of Proposition \ref{th:opt}]
We abbreviate $p=\frac{2(n+1)}{n+3}$ and let $\delta$ be a function as in Proposition \ref{th:opt}. Lemma \ref{lem:Pkd} with $A=-\Delta_g$ implies that there is an upper bound corresponding to the (small) spectral cluster $[(\kappa-\delta(\kappa))^2,(\kappa+\delta(\kappa))^2]$ of the form
\begin{align*}
\norm[L^p\to L^2]{P_{\kappa,\delta(\kappa)}} \le 2\delta(\kappa)^{1/2} \kappa^{1/2} \left( 1+\epsilon_1(\kappa)\right)^{\frac12}
\norm[L^p\to L^{p'}]{\left(-\Delta_g-(\kappa+\ii\delta(\kappa))^2\right)^{-1}}^{1/2}\,.
\end{align*}
with $\epsilon(\kappa) = \delta(\kappa)/\kappa$ and $\epsilon_1(\kappa) = \epsilon(\kappa) + \epsilon(\kappa)^2/2$. On the other hand, the proof of optimality of Sogge's (unit size) spectral cluster estimates \cite{Sogge1988,Sogge1989} shows that there is a constant $c>0$ such that for all sufficiently large $\kappa$,
\begin{align*}
\kappa^{-\frac{n-1}{n+1}}\big\|P_{\kappa,\frac12}\big\|_{L^p\to L^2}^2 \ge c \kappa^{-n+1} \left(N((\kappa+1/2)^2) - N((\kappa-1/2)^2)\right),
\end{align*}
where $N(\lambda)$ is the number of eigenvalues of $-\Delta_g$ less than $\lambda$, counting multiplicities.

These facts are true on any compact Riemannian manifold. We now assume that $(M,g)$ is a Zoll manifold and, without loss of generality after rescaling the metric, that the common minimal period of all the geodesics of $M$ is $2\pi$. As shown by Weinstein \cite{Weinstein1977} there is a constant $\alpha$ such that all the eigenvalues of $-\Delta_g$ cluster around the values $(k+\alpha)^2$ for $k\in\N$. To be more precise, there is a constant $C$ depending on $M$ such that each non-zero eigenvalue $\lambda_j$ of $-\Delta_g$ is in a cluster $[(k+\alpha-C/k)^2, (k+\alpha+C/k)^2]$ for some $k\in \N$. As a consequence, for $\kappa=k+\alpha$ and  $\delta\ge C/k$, we have $P_{\kappa,\delta}=P_{\kappa,\frac12}$. Moreover, we have by the sharp Weyl law and the clustering property
$$
c' = \liminf_{\kappa=k+\alpha\to\infty} \kappa^{-n+1} \left(N((\kappa+1/2)^2) - N((\kappa-1/2)^2)\right)> 0 \,. 
$$
Therefore, by combining these results,
\begin{align*}
&\limsup_{|\kappa|\to\infty}\left|\left(\kappa+\ii\delta(\kappa)\right)^2\right|^{\frac{1}{n+1}}\norm[L^p\to L^{p'}]{\left(-\Delta_g-(\kappa+\ii\delta(\kappa))^2\right)^{-1}}\\
& \geq \limsup_{\kappa =k+\alpha\to\infty}\left|\left(\kappa+\ii\delta(\kappa)\right)^2\right|^{\frac{1}{n+1}} \norm[L^p\to L^{p'}]{\left(-\Delta_g-(\kappa+\ii\delta(\kappa))^2\right)^{-1}}\\
& \geq \limsup_{\kappa=k+\alpha\to\infty} \frac{\kappa^{\frac{2}{n+1}} \big(1+\epsilon(\kappa)^2\big)^{\frac{1}{n+1}}}{4\delta(\kappa)\kappa \left(1+\epsilon_1(\kappa)\right)} \norm[L^p\to L^2]{P_{\kappa,\delta(\kappa)}}^2\\
& = \limsup_{\kappa=k+\alpha\to\infty} \frac{\kappa^{-\frac{n-1}{n+1}}}{4\delta(\kappa)} \norm[L^p\to L^2]{P_{\kappa,1/2}}^2 \\
& \geq \frac{cc'}{4}\limsup_{\kappa=k+\alpha\to\infty} \frac{1}{\delta(\kappa)}\\
& = +\infty \,.
\end{align*}
This proves Proposition \ref{th:opt}. 
\end{proof}

\section*{Appendix}

In this appendix we explain why the bounds in Lemmas \ref{lem:dyadic}, \ref{lem:Tadd} and \ref{lem:S}, which are proved in \cite{DosSantos2014} for $n\geq 3$, remain valid in $n=2$.

The Hadamard parametrix is of the form \eqref{eq:defT} where 
\begin{align*}
F(x,y,z)=\sum_{j=0}^N \alpha_j(x,y) F_j(d_g(x,y),z)
\end{align*} 
with smooth coefficients $\alpha_j$ and the Bessel potentials
\begin{align*}
F_j(r,z)&= \frac{j!}{(2\pi)^n}\int_{\R^n}\frac{\e^{\ii x\xi}}{(|\xi|^2-z)^{1+j}}\dd \xi
=c_jr^{-\frac n2+j+1}z^{\frac n4 -\frac{j+1}{2}}K_{n/2-1-j}(-\ii\sqrt{z}r)\,.
\end{align*}
For the second identity, we note that $\im\sqrt{z}>0$ and refer to \cite[p. 288]{Gelfand1964}. 
The behaviour of $F_j$ can be analysed by studying the Bessel functions $K_{\rho}$, which can be written as
\begin{align}
K_\rho(w)=\int_0^\infty \e^{-w\cosh(t)}\cosh(\rho t)\dd t\,.
\label{eq:K} 
\end{align}
Since the bounds for $n\ge 3$ appear in \cite[Lemma 3.1]{DosSantos2014}, we will now specialize to the case $n=2$. Most of the necessary bounds on $K_\rho$ can be found, for instance, in \cite{Kenig1987}. A special case of the bound for $j=0$ also appears in \cite[Lemma 4.3]{Sogge1988} but we will need a more precise estimate.

\begin{lemma}\label{lem:Fnu}
Let $\delta>0$ and $n=2$. The following bounds hold for all $|z|\ge\delta$.
\begin{enumerate}[a)]
\item For $|z|^{1/2}r\le 1$ it holds that
\begin{align}
|F_j(r,z)|\le \begin{cases}
C_0  \left( 1+ \log(|z|^{-\frac12}r^{-1})\right) &\text{if }j=0\\
C_j r^{2j}&\text{if }j\ge 1\,.
\end{cases} 
\label{eq:T02} 
\end{align}
\item  For $|z|^{1/2}r\ge 1$  it holds that
\begin{align}
F_j(r,z)=|z|^{\frac{1}{4}-\frac{j+1}{2}}\e^{\ii\sqrt{z}r}r^{-\frac{1}{2}+j}a_j(r,z)
\label{eq:Tnu2} 
\end{align}
with smooth functions $a_j$ satisfying
\begin{align*}
\left|\frac{\partial^\alpha a_j}{\partial r^\alpha}(r,z)\right|\le C_{\alpha,j}r^{-\alpha}\,.
\end{align*}
\end{enumerate}
\end{lemma}

\begin{proof}
From \eqref{eq:K} it is possible to prove that for $|w|\le 1$ and $\re w>0$, 
\begin{align*}
|\e^{\rho^2}\rho K_\rho(w)|\le C|w|^{-|\re \rho|}
\end{align*}
which appears as equation $(2.23)$ in \cite{Kenig1987}. This proves the first part of Lemma \ref{lem:Fnu} for $j\ge1$. For $j=0$, we use the representation
\begin{align*}
K_0(w)=\int_1^\infty\frac{\e^{-w u}}{\sqrt{u^2-1}}\dd u
\end{align*} 
as well as integration by parts to prove $|K_0(w)|\le C(1+\log|w|^{-1})$. 
The second part of Lemma \ref{lem:Fnu} follows from the corresponding bounds on $K_\rho$ as proven in $(2.26)$ and $(2.27)$ in \cite{Kenig1987}.
\end{proof}

The main difference to the case $n\ge3$ is the appearance of a logarithm in the bound on $F_0$. We will show how this changes the proofs of Lemmas \ref{lem:dyadic}, \ref{lem:Tadd} and \ref{lem:S} given in \cite{DosSantos2014}. 

The part of Lemma \ref{lem:dyadic} for $n=2$ concerning $T_0(z)$ follows from \eqref{eq:T02}. The bounds on $T_\nu(z)$ for $\nu\ge 1$ given in Lemma \ref{lem:dyadic} can be proved in the same way as in the case $n\ge 3$ in \cite{DosSantos2014} by applying the Carleson--Sj{\"olin} theorem and using \eqref{eq:Tnu2}. 

With regards to Lemma \ref{lem:Tadd}, note that by \eqref{eq:Tnu2} the same proof as in \cite{DosSantos2014} implies that $T(z)-T_0(z)$ can be bounded by 
\begin{align*}
\norm[q]{(T(z)-T_0(z))u}\lesssim |z|^{-\frac14-\frac{1}{q}}\norm[p]{u}
\end{align*}
for $p>4/3$ and $q\ge 3p'$. Using the already established bound on the kernel of $T_0(z)$, Young's inequality yields as in the proof of Proposition \ref{prop:weak} that for $1/p-1/q<1$
\begin{align*}
\norm[q]{T_0(z)u}\lesssim |z|^{\frac{1}{p}-\frac{1}{q}-1} \norm[p]{u} \,.
\end{align*}

Finally, for Lemma \ref{lem:S} we first note that the remainder $S(z)$ in \eqref{eq:T} is the sum of two operators given by
\begin{align*}
(S_1u)(z)= - \int_M \left( 2(\nabla_g\tilde\chi)(x)\cdot_g\nabla_g + (\Delta_g\tilde\chi)(x)\right) F(x,y,z)\chi(y)u(y)\dd\mu_g(y)
\end{align*} 
and 
\begin{align*}
(S_2(z)u)(x)=-\int_M \tilde\chi(x)(\Delta_{g,x}\alpha_N)(x)F_N\big(d_g(x,y),z\big)\chi(y)u(y)\dd\mu_g(y)\,.
\end{align*}
Since $\alpha_N$ is smooth and since for $N\geq 1$ and $|z|\ge\delta$,
\begin{align*}
|F_N(r,z)|\lesssim |z|^{-\frac34} \,,
\end{align*}
we infer that $S_2(z)$ is a bounded operator from $L^p(M)$ to $L^q(M)$ with norm bounded by $|z|^{-3/4}$. In order to bound $S_1$ we note that by the choice of $\chi$ and $\tilde\chi$ there is an $\epsilon>0$ such that $\chi(y)\nabla_g\tilde\chi(x)$ and $\chi(y)\Delta_g\tilde\chi(x)$ vanish when $d_g(x,y)\leq\epsilon$. From \eqref{eq:Tnu2} we can thus conclude that the kernel of $S_1(z)$ is of the form
\begin{align*}
|z|^{\frac14}\e^{\ii\sqrt{z}d_g(x,y)}b(x,y,z)
\end{align*}
with a smooth function $b$. Lemma \ref{lem:S} can then be proved in the same way as in \cite{DosSantos2014} by means of the Carleson--Sj{\"o}lin theorem.



\bibliographystyle{amsalpha}

\end{document}